\documentclass[12pt]{amsart}

\usepackage[utf8]{inputenc}

\usepackage[dvipsnames]{xcolor}

\usepackage{amsmath, amssymb, amsthm}
\usepackage{bbm}
\usepackage[giveninits=true, style=alphabetic]{biblatex}
\usepackage{blkarray}
\usepackage{cancel}
\usepackage{enumitem}
\usepackage{euscript}
\usepackage{float}
\usepackage{graphicx}
\usepackage{mathrsfs}
\usepackage{mathtools}
\usepackage{microtype}
\usepackage{ragged2e}
\usepackage{thmtools}
\usepackage{tikz}
\usepackage[Symbolsmallscale]{upgreek}

\usepackage{custom}

\usepackage{hyperref}
\usetikzlibrary{cd}

\hypersetup{citecolor=red, colorlinks=true, linkcolor=blue}

\allowdisplaybreaks{}

\title[Dimension-free log-Sobolev inequalities]{Dimension-free log-Sobolev inequalities for mixture distributions}
\author{Hong-Bin Chen \and Sinho Chewi \and Jonathan Niles-Weed}
\date{\today}

\address[Hong-Bin Chen]{Courant Institute of Mathematical Sciences, New York University, New York, NY, USA}
\email{hbchen@cims.nyu.edu}

\address[Sinho Chewi]{Department of Mathematics, Massachusetts Institute of Technology, Cambridge, MA, USA}
\email{schewi@mit.edu}

\address[Jonathan Niles-Weed]{Courant Institute of Mathematical Sciences, New York University, New York, NY, USA}
\email{jnw@cims.nyu.edu}

\declaretheorem[name=Corollary]{cor}

\declaretheorem[name=Lemma]{lem}
\declaretheorem[name=Proposition]{prop}
\declaretheorem[name=Remark, style=remark]{rmk}
\declaretheorem[name=Theorem]{thm}
\declaretheorem[name=Theorem, numbered=no]{thm*}

\begin{document}

\maketitle

\begin{abstract}
    We prove that if ${(\PP_x)}_{x\in \sX}$ is a family of probability measures which satisfy the log-Sobolev inequality and whose pairwise chi-squared divergences are uniformly bounded, and $\mu$ is any mixing distribution on $\sX$, then the mixture $\int \PP_x \, \D \mu(x)$ satisfies a log-Sobolev inequality. In various settings of interest, the resulting log-Sobolev constant is dimension-free. In particular, our result implies a conjecture of Zimmermann and Bardet et al.\ that Gaussian convolutions of measures with bounded support enjoy dimension-free log-Sobolev inequalities. 
\end{abstract}


\section{Introduction}

Functional inequalities, such as the Poincar\'e inequality and the log-Sobolev inequality, have played a key role in the study of subjects such as concentration of measure and quantitative convergence analysis of Markov processes~\cite{bakrygentilledoux2014, vanhandel2018probability} (in particular for spin systems~\cite{martinelli1999spinsystems, weitz2004thesis}), as well as the geometry of metric measure spaces~\cite{ledoux2000markovdiffusion}. It is therefore of considerable interest to identify situations in which such inequalities hold, and furthermore to identify simple criteria which imply their validity.

We begin with a few motivating examples. Suppose that $\mu$ is a probability measure on $\R^d$ whose support is contained in the Euclidean ball of radius $R$, and let $\gamma_{0,t}$ denote the centered Gaussian distribution with variance $tI_d$.
What functional inequalities can we expect the convolution measure $\mu * \gamma_{0,t}$ to satisfy? This question, motivated by random matrix theory, was initiated in~\cite{zimmermann1, zimmermann2}, and further investigated in~\cite{wang2016functional, bardetetal2018gaussianconv}. These works prove that $\mu * \gamma_{0,t}$ satisfies both a Poincar\'e inequality and a log-Sobolev inequality; moreover, the Poincar\'e inequality holds with a constant depending only on $R$ and $t$, and not on the dimension $d$. Furthermore,~\cite{bardetetal2018gaussianconv} conjectures that the same holds true for the log-Sobolev constant, and they verify the conjecture in a number of special cases.

The sharp dimension dependence of the log-Sobolev inequality for Gaussian convolutions is of particular interest due to numerous recent applications in non-convex optimization and sampling; we refer to~\cite{chaudharietal2019entropysgd, blocketal2020langevinmanifold, blockmrougrahklin2020generative}.

Another line of work~\cite{chafaimalrieu2010mixtures, schlichting2019mixtures} studies the following question: let $\PP_0$ and $\PP_1$ be two probability measures on $\R^d$, and consider the mixture distribution $(1-p) \PP_0 + p\PP_1$ with mixing weight $p\in (0,1)$. If both $\PP_0$ and $\PP_1$ satisfy log-Sobolev inequalities, when does the mixture satisfy a log-Sobolev inequality too?

Although the two preceding examples may at first glance appear to be different in nature, we can in fact place them in the same framework, as follows. Let ${(\PP_x)}_{x\in \sX}$ be a family of probability measures satisfying the log-Sobolev inequality, and let $\mu$ be a mixture distribution on $\sX$; here, $\sX$ may be finite or infinite. When does the mixture $\int \PP_x \, \D \mu(x)$ satisfy a log-Sobolev inequality?
\begin{itemize}
    \item For the Gaussian convolution example, we take $\PP_x$ to be the Gaussian distribution with mean $x$ and variance $tI_d$.
    \item For the mixture example, we take $\mu$ to be the Bernoulli distribution with parameter $p$.
\end{itemize}

In this paper, we identify general conditions which ensure that a mixture distribution satisfies a log-Sobolev inequality. Our main contribution can be summarized as follows.

\begin{thm*}[informal]
    Let ${(\PP_x)}_{x\in \sX}$ be a family of probability measures satisfying the log-Sobolev inequality with a uniform constant $C_1$.
    Assume that the pairwise chi-squared divergences $\chi^2(\PP_x \mmid \PP_{x'})$ are uniformly bounded by $C_2$.
    Then, the mixture $\int \PP_x \, \D \mu(x)$ satisfies a log-Sobolev inequality with a constant depending only on $C_1$ and $C_2$.
\end{thm*}

In fact, in our main result, we will relax the assumption that the chi-squared divergences are uniformly bounded into a moment condition; see Theorem~\ref{thm:LS_new}. In turn, this will allow us to prove log-Sobolev inequalities for Gaussian convolutions of measures with sub-Gaussian tails, provided that the variance of the Gaussians is sufficiently large.

Crucially, the log-Sobolev constant has no dependence on the mixing distribution $\mu$. As we show in Section~\ref{scn:applications}, our general theorem yields dimension-free log-Sobolev inequalities in various settings; in particular, our result implies the conjecture of~\cite{zimmermann1, zimmermann2, bardetetal2018gaussianconv}.

The rest of the paper is organized as follows. In Section~\ref{scn:background}, we describe the setting of our general investigation and recall the definitions of a Poincar\'e inequality and a log-Sobolev inequality. We then state and prove our main theorem in Section~\ref{sect:main_new}.

In Section~\ref{scn:applications}, we illustrate our general result in a number of applications. Section~\ref{scn:gaussian_conv} is devoted to the proof of the aforementioned conjecture, and Sections~\ref{sec:sg_tails} and~\ref{sect:gen_diff} generalize the result to Gaussian convolutions of measures with sub-Gaussian tails and other diffusion semigroups. In Section~\ref{scn:mixtures_of_two}, we compare our results to prior work on functional inequalities for mixtures of two distributions. Then, in Section~\ref{scn:hypercube}, we discuss analogues of our result on the Boolean hypercube.


\section{Background and notation}\label{scn:background}

To state our results in a form that applies to both discrete and continuous mixture distributions, we adopt the general framework of~\cite{bakrygentilledoux2014} and let $\Gamma$ be a suitable notion of a gradient operator.
More precisely, let $\sY$ be a Polish space equipped with the Borel $\sigma$-algebra $\borelY$, and let $\mc A$ be a subspace of bounded measurable functions on $E$ containing all constant functions. Let $\Gamma:\mc A\times \mc A \to \mc A$ be a symmetric bilinear operator satisfying $\Gamma(f,f)\geq 0$ everywhere on $\sY$ for every $f\in \mc A$. In addition, we require $\Gamma$ to satisfy
\begin{align}\label{eq:Gamma_prop}
    \Gamma(1,1) =0\,,
\end{align}
where $1$ is understood as a constant function. 
For brevity, we write $\Gamma(f) = \Gamma(f,f)$. 

Important examples include the usual squared gradient $\Gamma(f) = \norm{\nabla f}^2$ on $\R^d$, and $\Gamma(f) = \sum_{i=1}^d {(D_i f)}^2$ on a product space $\sX = \eu X^d$, where $D_i f$ is the discrete gradient
\begin{align*}
    D_i f(x)
    &:= \sup_{x_i' \in \mc X} f(x_1,\dotsc,x_{i-1},x_i',x_{i+1},\dotsc,x_d) \\
    &\qquad{} - \inf_{x_i' \in \mc X} f(x_1,\dotsc,x_{i-1},x_i',x_{i+1},\dotsc,x_d)\,.
\end{align*}

For any probability measure $\rho$ on $(\sY, \borelY)$, we write
\begin{align*}
    \E_\rho [f] := \int_{\sY} f\,\D \rho
\end{align*}
for a $\rho$-integrable function $f$. In addition, we define
\begin{align*}
    \on{var}_\rho (f)
    &:= \E_\rho[{(f-\E_\rho)}^2]\,,\\
    \on{ent}_\rho (g)
    &:= \E_\rho(g\log g)-\E_\rho g\log \E_\rho g\,,
\end{align*}
for suitable measurable functions $f$ and $g$, with $g$ nonnegative. When there is no confusion, we often omit the brackets and parentheses in these expressions. If $X$ is a random variable with law $\mu$, we also write $\E f(X)=\E_\mu f$ and similarly for $\var$ and $\ent$.

We say that $\rho$ satisfies a Poinc\'are inequality (PI) if there is a constant $C$ such that
\begin{align}\tag{$\rm PI$} \label{eq:pi}
    \on{var}_\rho(f)\leq C\E_\rho\Gamma(f)\,,\qquad\forall f \in \mc A\,.
\end{align}
The optimal constant in this inequality is denoted $C_{\on{P}}(\rho)$. In addition, $\rho$ is said to satisfy a logarithmic Sobolev inequality (LSI) if there is a constant $C$ such that
\begin{align}\tag{$\rm LSI$}\label{eq:lsi}
    \on{ent}_\rho(f^2)\leq 2C\E_\rho\Gamma(f)\,,\qquad\forall f \in \mc A\,.
\end{align}
Similarly, we let $C_{\on{LS}}(\rho)$ denote the optimal constant in this inequality. 

For probability measures $\rho_1$ and $\rho_2$ on $(\sY,\borelY)$, the Kullback-Leibler (KL) divergence and the chi-squared divergence are defined as
\begin{align*}
    D_{\rm KL}(\rho_1 \mmid \rho_2)
    &:= \on{ent}_{\rho_2} \bigl( \frac{\D \rho_1}{\D \rho_2}\bigr) 
    = \int_\sY \frac{\D\rho_1}{\D \rho_2} \ln \frac{\D \rho_1}{\D \rho_2} \, \D \rho_2 
    = \int_{\sY} \bigl(\ln \frac{\D \rho_1}{\D \rho_2}\bigr) \, \D \rho_1\,, \\
    \chi^2(\rho_1\mmid \rho_2)
    &:= \on{var}_{\rho_2} \bigl( \frac{\D \rho_1}{\D \rho_2}\bigr)
    = \int_\sY \bigl(\frac{\D \rho_1}{\D\rho_2}-1\bigr)^2 \, \D\rho_2
    = \int_\sY \frac{\D \rho_1}{\D \rho_2} \, \D \rho_1 - 1\,.
\end{align*}
The expressions above are understood to be $+\infty$ if $\rho_1$ is not absolutely continuous w.r.t.\ $\rho_2$.

\section{Main theorem}\label{sect:main_new}


In addition to $(\sY,\borelY)$, let $\sX$ be a polish space with Borel $\sigma$-algebra $\borelX$. We consider a Markov kernel $\PP : \sX\times \borelY \to [0,1]$ satisfying: (1) for each $x\in \sX$, $\PP(x,\cdot)$ is a probability measure on $(\sY,\borelY)$, and (2) for each $B\in \mc \borelY$, $\PP(\cdot,B)$ is a $\borelX$-measurable function on $\sX$. We also write $\PP_x := \PP(x,\cdot)$ for convenience. This kernel naturally induces a transition map which maps bounded measurable functions on $\sX$ to bounded measurable functions on $\sY$:
\begin{align*}
    Pf(x) := \int_{\sX} f \, \D\PP_x\,,\qquad\forall x\in \sX\,.
\end{align*}
For a probability measure $\mu$ on $(\sX,\borelX)$, we denote by $\mu \PP$ the probability measure on $(\sY, \borelY)$ defined by the duality
\begin{align*}
    \int_\sY f \,\D\mu \PP = \int_\sX \PP f\, \D \mu\,.
\end{align*}

Lastly, we introduce the following quantities.
\begin{align}
    K_{\rm P}(\PP;\mu)
    &:= \esssup_{\text{$\mu$-a.s.}\ x\in \sX}C_{\on{P}}(\PP_x)\,, \label{eq:K_P_new}\\
    K_{\on{LS}}(\PP;\mu)
    &:= \esssup_{\text{$\mu$-a.s.}\ x\in \sX}C_{\on{LS}}(\PP_x)\,,\label{eq:K_LS_new}\\
    K_{p,\,\chi^2}(\PP;\mu)
    &:= {\E\bigl[{\bigl(1+\chi^2(P_X\mmid P_{X'})\bigr)}^p\bigr]}^{\frac{1}{p}}\,,
    \label{eq:K_chi^2_new}
\end{align}
for $p\geq 1$, where $X$ and $X'$ are i.i.d.\ with law $\mu$.
Since~\eqref{eq:lsi} implies~\eqref{eq:pi} with the same constant, we have $K_{\rm P}(\PP;\mu) \le K_{\rm LS}(\PP;\mu)$. Throughout, for $p\geq1$, we set $p^*=\frac{p}{p-1}$ to be the dual exponent.

\begin{thm}\label{thm:LS_new}
\leavevmode
\begin{enumerate}
    \item \label{item:P} If $K_{\rm P}(\PP;\mu)$ and $K_{p,\,\chi^2}(\PP;\mu)$ are finite for some $p>1$, then $\mu \PP$ satisfies~\eqref{eq:pi} with constant
\begin{equation*}
    C_{\on{P}}(\mu \PP) \leq K_{\rm P}\, \{p^*+ \, K_{p,\,\chi^2}^{p^*}\}\,,
\end{equation*} 
where $ K_{\rm P}= K_{\rm P}(\PP;\mu)$ and $K_{p,\,\chi^2}=K_{p,\,\chi^2}(\PP;\mu)$.
    \item \label{item:LS} If $K_{\rm LS}(\PP;\mu)$ and $K_{p,\,\chi^2}(\PP;\mu)$ are finite for some $p>1$, then $\mu \PP$ satisfies~\eqref{eq:lsi} with constant
    \begin{align*}
        C_{\rm LS}(\mu \PP)
        &\le 3K_{\rm LS} \, (p^* + K_{p,\;\chi^2}^{p^*}) \, (1 + \log K_{p, \;\chi^2}^{p^*})\,,
    \end{align*}
where $ K_{\rm LS}= K_{\rm LS}(\PP;\mu)$ and $K_{p,\,\chi^2}=K_{p,\,\chi^2}(\PP;\mu)$.
\end{enumerate}

\end{thm}

\begin{rmk}
Our theorem is stated with a simpler constant for readability.
A slightly sharper constant can be read off from the proof. Our results clearly extend to the case $p=\infty$ ($p^*=1$) with 
\begin{align*}
    K_{\infty,\chi^2}(\PP;\mu) := 1+ \esssup_{\text{$\mu$-a.s.}\ x,x'\in \sX}\chi^2(\PP_x\mmid \PP_{x'})\,.
\end{align*}
\end{rmk}


For both steps, our starting point is to apply classical decompositions for the variance and the entropy, which have been used to prove functional inequalities for spin systems (see e.g.\ the appendix of~\cite{weitz2004thesis}). If $X$ is a random variable drawn according to $\mu$, then
\begin{align}
    \var_{\mu\PP} f
    &= \E \var_{\PP_X} f + \var \E_{\PP_X} f\,, \label{eq:decomp_pi_new} \\
    \ent_{\mu\PP} f^2
    &= \E \ent_{\PP_X} f^2 + \ent \E_{\PP_X} f^2\, \label{eq:decomp_lsi_new}.
\end{align}
In both of these decompositions, the first term is easy to handle because we can apply the PI, resp.\ LSI, for the family ${(\PP_x)}_{x\in\sX}$ inside the expectation.
The crux of the proof is therefore the second terms.

\begin{proof}[Proof of Theorem~\ref{thm:LS_new}~\eqref{item:P}]
    In the case $p = \infty$ (i.e., the pairwise chi-squared divergences are uniformly bounded), the Poincar\'e inequality can be proven via a straightforward generalization of~\cite{bardetetal2018gaussianconv}. However, the case $1 < p < \infty$ requires non-trivial modifications, and we present a complete proof.

Let $X$ be a random variable with law $\mu$.
As described above, we use the decomposition~\eqref{eq:decomp_pi_new}, and we focus on the problematic second term
\begin{align*}
    \var \E_{\PP_X} f
    &=  \E[\abs{\E_{\PP_X} f - \E_{\mu\PP} f}^2]\,.
\end{align*}
We can write
\begin{align*}
    \E_{\PP_X} f - \E_{\mu\PP} f &= \int f \, \Bigl(1 - \frac{\D \mu \PP}{\D \PP_X}\Bigr) \, \D \PP_X\\
    & = -\int f\, \Bigl(1 - \frac{\D \PP_{X}}{\D \mu\PP}\Bigr)\, \D \mu\PP\,.
\end{align*}
For brevity, we write $\chi^2_{\rho,\,\rho'} := \chi^2(\rho\mmid \rho')$.
Applying the Cauchy-Schwarz inequality to the above display, we have
\begin{align*}
    \var \E_{\PP_X} f
    &\leq \E \min\{(\var_{\mu\PP} f) \, \chi^2_{\PP_{X},\, \mu\PP}, \; (\var_{\PP_X} f) \,\chi^2_{\mu\PP,\, \PP_X}\}\\
    &\leq \E\bigl[ {(\var_{\mu\PP} f)}^{1/p} \, {(\chi^2_{\PP_{X},\, \mu\PP})}^{1/p} \, {(\var_{\PP_X} f)}^{1/p^*} \, {(\chi^2_{\mu\PP,\, \PP_X})}^{1/p^*}\bigr]\,.
\end{align*}
Then, Young's inequality implies that for all $\lambda > 0$,
\begin{align*}
   \var \E_{\PP_X} f\leq \frac{\lambda^p}{p} \, (\var_{\mu\PP}f) \E\bigl[(\chi^2_{\PP_{X},\, \mu\PP}) \, {(\chi^2_{\mu\PP,\, \PP_X})}^{p-1}\bigr] + \frac{\lambda^{-p^*}}{p^*}\E \var_{\PP_X} f\,.
\end{align*}
Setting
\begin{align*}
    \lambda = {\E\bigl[(\chi^2_{\PP_{X},\, \mu\PP}) \, {(\chi^2_{\mu\PP,\, \PP_X})}^{p-1}\bigr]}^{-\frac{1}{p}}
\end{align*}
and substituting the above into \eqref{eq:decomp_pi_new} yields
\begin{align*}
    \var_{\mu\PP}f \leq \bigl\{p^* + {\E\bigl[(\chi^2_{\PP_{X},\, \mu\PP}) \, {(\chi^2_{\mu\PP,\, \PP_X})}^{p-1}\bigr]}^{\frac{1}{p-1}}\bigr\}\E\var_{\PP_X} f\,.
\end{align*}
Using H\"older's inequality and the convexity of the chi-squared divergence, we can see
\begin{align*}
    \E\bigl[(\chi^2_{\PP_{X},\, \mu\PP}) \, {(\chi^2_{\mu\PP,\, \PP_X})}^{p-1}\bigr]
    &\leq \E\bigl[{(\chi^2_{\PP_X,\;\PP_{X'}})}^p\bigr]
\end{align*}
where $X'$ is an i.i.d.\ copy of $X$. The desired result follows from the definitions of $K_{\on{P}}(\PP;\mu)$ in \eqref{eq:K_LS_new} and $K_{p,\,\chi^2}(\PP;\mu)$ in \eqref{eq:K_chi^2_new}.
\end{proof}

To prove the second assertion in Theorem~\ref{thm:LS_new}, we derive a so-called defective LSI for $\mu \PP$, which can be tightened to yield a full LSI\@. In order to control the second term in \eqref{eq:decomp_lsi_new}, we need a lemma.

\begin{lem}\label{lemma:ineq_new}
Let $\pi$ and $\rho$ be two probability measures. Then, the following holds for every non-negative function $f$: 
\begin{equation*}
    \E_\pi f \log \frac{\E_\pi f}{\E_\rho f}
    \leq \on{ent}_\pi(f) + \E_\pi(f) \log \bigl(1+\chi^2(\pi\mmid \rho)\bigr)\,,
\end{equation*}
where by convention both sides vanish if $\E_\pi f = 0$.
\end{lem}

\begin{proof}
Recall the Donsker--Varadhan theorem\footnote{See~\cite[Theorem 5.4]{rassoulaghaseppalainen2015largedeviations} or~\cite[Lemma 6.2.13]{dembozeitouni2010largedeviations}.}: for any probability measures $\mu$ and $\nu$, it holds
\begin{equation}\label{eq:dv_fact}
    D_{\rm KL}(\mu \mmid \nu) = \sup_{g}{\{\E_\mu g - \log \E_\nu \exp(g)\}}\,,
\end{equation}
where the supremum is taken over all $g$ for which the expectations on the right side make sense.

We may assume that $\pi$ is absolutely continuous with respect to $\rho$ and that $\E_\pi(f \log f) < \infty$; otherwise, the expression on the right side is infinite.
We may therefore assume that $0 < \E_\pi f < \infty$, and, since each term in the lemma statement is homogeneous in $f$, we may assume without loss of generality that $\E_\pi f = 1$.

Define a new probability measure $\pi_f$ by $\frac{\D \pi_f}{\D \pi} = f$.
Then,
\begin{align*}
    \E_\pi\bigl[f \log \frac{f}{\E_\rho f}\bigr] = \E_{\pi_f} \log \frac{f}{\E_\rho f}
    &\leq D_{\rm KL}(\pi_f \mmid \rho) + \log \E_\rho \exp \log \frac{f}{\E_\rho f}\\
    &= D_{\rm KL}(\pi_f \mmid \rho)\,,
\end{align*}
where we have used~\eqref{eq:dv_fact}.
Since
\begin{equation*}
    D_{\rm KL}(\pi_f \mmid \rho) = \E_\pi\bigl[f \log\bigl(f\, \frac{\D \pi}{\D \rho}\bigr)\bigr]\,,
\end{equation*}
subtracting $\E_\pi(f \log f)$ from both sides of the inequality above and recalling that we have assumed that $\E_\pi f = 1$ yields
\begin{equation*}
    \E_\pi f \log \frac{\E_\pi f}{\E_\rho f}
    \leq \E_\pi \bigl[f \log \frac{\D \pi}{\D \rho}\bigr]\,.
\end{equation*}

Continuing, we have again by \eqref{eq:dv_fact} that
\begin{align*}
    \E_\pi\bigl[f \log \frac{\D \pi}{\D \rho}\bigr] = \E_{\pi_f} \log \frac{\D \pi}{\D \rho} &\leq D_{\rm KL}(\pi_f \mmid \pi) + \log \E_\pi \exp \log \frac{\D \pi}{\D \rho}\\
    &= \E_\pi(f \log f) + \log\bigl(1+\chi^2(\pi\mmid \rho)\bigr)\,,
\end{align*}
as claimed.
\end{proof}

\begin{proof}[Proof of Theorem~\ref{thm:LS_new}~\eqref{item:LS}]
Let $X,X'$ be i.i.d.\ copies with law $\mu$. 
The second term $\ent \E_{P_X}(f^2)$ in \eqref{eq:decomp_lsi_new} can be written as
\begin{align*}
   \ent \E_{P_X}(f^2)= \E\Bigl[ \E_{P_{X}}(f^2) \log\frac{\E_{P_{X}}(f^2)}{\E_{\mu P}(f^2)}\Bigr]\,.
\end{align*}
Setting $\pi = P_X$ and $\rho =\mu P$ in Lemma~\ref{lemma:ineq_new}, we obtain
\begin{align}\label{eq:decom_LS_2}
    \ent \E_{P_X}(f^2)
    \leq \E \on{ent}_{P_X}(f^2)+ \E\bigl[\E_{P_X}(f^2) \log\bigl(1 + \chi^2(P_X \mmid P_{X'})\bigr) \bigr]
\end{align}
where we also used the convexity of the chi-squared divergence in the second inequality. 
The definition of $K_{p,\,\chi^2}(P;\mu)$ in \eqref{eq:K_chi^2_new} ensures that
\begin{align*}
    \E \exp\Bigl( p\log \bigl(1+\chi^2(\PP_X\mmid\PP_{X'})\bigr)- p \log K_{p,\,\chi^2}(P;\mu)\Bigr) \leq 1\,.
\end{align*}
Using the variational principle for the entropy~\cite[Lemma 3.15]{vanhandel2018probability}:
\begin{align*}
    \ent Y
    &= \sup\{\E(YZ) \mid Z~\text{is a random variable with}~\E\exp Z \le 1\}\,,
\end{align*}
we obtain
\begin{align*}
    &\E\bigl[\E_{P_X}(f^2) \log\bigl(1 + \chi^2(P_X\mmid P_{X'})\bigr) \bigr] \\
    &\qquad \leq \frac{1}{p}\ent \E_{\PP_X}(f^2)+ \log K_{p,\,\chi^2}(P;\mu) \E_{\mu\PP}(f^2)\,.
\end{align*}
Substituting this into~\eqref{eq:decom_LS_2} yields
\begin{align*}
    \ent \E_{\PP_X}(f^2) \leq p^* \, \bigl\{\E \ent_{\PP_X}(f^2)+\log K_{p,\,\chi^2}(P;\mu) \E_{\mu\PP}(f^2)\bigr\}\,.
\end{align*}
We insert this into \eqref{eq:decomp_lsi_new} to obtain:
\begin{align*}
    \ent_{\mu\PP}(f^2)
    &\le (p^* + 1) \E\ent_{\PP_X}(f^2) + p^* \log K_{p,\,\chi^2}(P;\mu) \E_{\mu\PP}(f^2) \\
    &\leq 4p^* K_{\on{LS}}(P;\mu)\E_{\mu\PP}\Gamma(f)+ p^*\log K_{p,\,\chi^2}(P;\mu)\E_{\mu\PP}(f^2)\,.
\end{align*}
This inequality is known as a \emph{defective LSI} (see~\cite[\S 5]{bakrygentilledoux2014}). It is standard that a defective LSI together with a Poincar\'e inequality implies a full LSI; this is known as \emph{tightening} the LSI, and we refer to Appendix~\ref{scn:tightening_lsi}
 for details. Together with the PI in the first assertion of Theorem~\ref{thm:LS_new}, this completes the proof.
\end{proof}

\section{Applications}\label{scn:applications}


\subsection{Gaussian convolutions}\label{scn:gaussian_conv}

Set $\sY =\R^d$, $\mc A = C_{\rm b}^\infty(\R^d)$ (infinitely differentiable functions with bounded derivatives), and $\Gamma(f) = \norm{\nabla f}^2$. Let $\mu$ be a probability measure supported on $B(0,R):=\{x\in\R^d:\|x\|\leq R\}$, and for $x\in \R^d$ and $t>0$, let \begin{align*}
    \gamma_{x,t}( y)=\frac{1}{{(2\pi t)}^\frac{d}{2}} \, e^{-\frac{\|y-x\|^2}{2t}}
\end{align*}
be the Gaussian with mean $x$ and variance $tI_d$.
If we take $\PP_x = \gamma_{x,t}$, then the measure $\mu\PP$ is the convolution $\mu * \gamma_{0,t}$.

Functional inequalities for the measure $\mu * \gamma_{0,t}$ were studied in~\cite{zimmermann1, zimmermann2}, and further investigated in~\cite{wang2016functional, bardetetal2018gaussianconv}. In particular,~\cite{bardetetal2018gaussianconv} proves that $C_{\rm P}(\mu * \gamma_{0,t})$ is bounded above by a function of $R$ and $t$, and is therefore dimension-free.

For the log-Sobolev constant, these works also show that $C_{\rm LS}(\mu * \gamma_{0,t})$ is finite, but the precise dependence of this constant (in particular on the dimension) was previously unknown. Bardet et al.~\cite{bardetetal2018gaussianconv} verify in several cases that $C_{\rm LS}(\mu * \gamma_{0,t})$ is dimension-free, and they conjecture that this is true in general. We now show that their conjecture is an immediate consequence of Theorem~\ref{thm:LS_new}.

It is well-known that $\gamma_{x,t}$ satisfies~\eqref{eq:lsi} with $C_{\rm LS}(\gamma_{x,t}) = t$.
Also, for $x,x'\in \R^d$ and $t\geq 0$, a straightforward computation shows that
\begin{align*}
    \chi^2(\gamma_{x,t}\mmid \gamma_{x',t}) = e^{\|x-x'\|^2/t}-1\,.
\end{align*}
Hence, $K_{\infty,\;\chi^2}(\PP;\mu) \le e^{4R^2/t}$ and we deduce the following result.

\begin{cor}\label{thm:LS_gausssian}
Let $\mu$ be a probability measure on $\R^d$ supported on $B(0,R)$ for some $R\geq 0$.
Then, for each $t\geq 0$, $\mu*\gamma_{0,t}$ satisfies~\eqref{eq:lsi} with
\begin{align*}
    C_{\rm LS}(\mu*\gamma_{0,t}) \leq 6 \, (4R^2+t) \, e^\frac{4R^2}{t}\,.
\end{align*}
\end{cor}

Bardet et al.\ also prove that $\mu * \gamma_{0,t}$ satisfies a $T_2$ transport-entropy inequality with a dimension-dependent constant; see~\cite[\S 9.3]{villani2003topics} for the relevant background. Since a log-Sobolev inequality implies a $T_2$ inequality with the same constant~\cite{ottovillani2000lsi}, we immediately obtain the following improvement.

\begin{cor}
    Let $\mu$ be a probability measure on $\R^d$ supported on $B(0,R)$ for some $R\geq 0$.
    Then, for each $t\geq 0$, $\mu*\gamma_{0,t}$ satisfies a $T_2$ transport-entropy inequality with constant
\begin{align*}
    C_{T_2}(\mu*\gamma_{0,t}) \leq 6 \, (4R^2+t) \, e^\frac{4R^2}{t}\,.
\end{align*}
\end{cor}

\begin{rmk}
    These results show that evolving a compactly supported measure for a short time under the heat flow yields dimension-free functional inequalities, which can be interpreted as a strong regularizing effect of the heat flow. This is in line with other results on the smoothing behavior of the heat flow, e.g.~\cite{eldanlee2018regularization}.
\end{rmk}
\begin{rmk}[sharpness of the result] \label{rmk:improving_asymp_const}
    As $t\to\infty$, Corollary~\ref{thm:LS_gausssian} implies that
    \begin{align*}
        \limsup_{t\to\infty} \frac{C_{\rm LS}(\mu * \gamma_{0,t})}{t} \le 6\,.
    \end{align*}
    It is easy to improve this to $1$, which is sharp. Indeed, from the subadditivity of the log-Sobolev constant under convolution, for $t \ge 4R^2$,
    \begin{align*}
        C_{\rm LS}(\mu * \gamma_{0,t})
        &\le C_{\rm LS}(\mu * \gamma_{0,4R^2}) + C_{\rm LS}(\gamma_{0,t-4R^2})
        \le t + 130\,R^2\,.
    \end{align*}
    
    On the other hand, as $t\searrow 0$, the exponential dependence on $R^2/t$ cannot be avoided, as a simple example shows. Indeed, consider the measure $\mu = \frac{1}{2} \,\delta_{-R} + \frac{1}{2}\, \delta_R$ in one dimension and $0 < t \ll R$. Define the function $f : \R\to [-1,1]$ via
    \begin{align*}
        f(x)
        &:= \begin{cases} -1~\text{for}~x < -R/2, \\ +1~\text{for}~x > +R/2, \\ \text{linear interpolation in between}. \end{cases}
    \end{align*}
    Let $g$ denote a standard Gaussian variable.
    Then, $\E_{\mu * \gamma_{0,t}} f = 0$, so
    \begin{align*}
        \var_{\mu * \gamma_{0,t}} f
        &= \E_{\mu * \gamma_{0,t}}(f^2) \\
        &\ge \frac{1}{2} \, \Pr\bigl\{-R + \sqrt t \, g \le - \frac{R}{2}\bigr\} + \frac{1}{2} \, \Pr\bigl\{R + \sqrt t \,  g \ge \frac{R}{2}\bigr\} \\
        &= \Pr\bigl\{g \le \frac{R}{2\sqrt t}\bigr\}
        \ge \frac{1}{2}\,.
    \end{align*}
    On the other hand, $\abs{f'} = 2/R$ on $[-R/2, R/2]$, so
    \begin{align*}
        \E_{\mu * \gamma_{0,t}}(\abs{f'}^2)
        &\le \frac{4}{R^2} \, \Pr\bigl\{ g \ge \frac{R}{2\sqrt t}\bigr\}
        \le \frac{2}{R^2} \exp\bigl(-\frac{R^2}{8t}\bigr)\,,
    \end{align*}
    by standard Gaussian tail bounds. This yields the following lower bound on the Poincar\'e constant of $\mu * \gamma_{0,t}$:
    \begin{align*}
        C_{\rm P}(\mu * \gamma_{0,t})
        &\ge \frac{1}{4} \, R^2 \exp \frac{R^2}{8t}\,.
    \end{align*}
    Hence, the exponential dependence on $R^2/t$ is already present in the Poincar\'e constant. However, it is worth noting that the $\exp(4R^2/t)$ dependence in the log-Sobolev constant enters \emph{only} via the Poincar\'e constant through the method of tighening a defective log-Sobolev inequality. In particular, if $\mu$ is known \emph{a priori} to satisfy a Poincar\'e inequality with constant $C_{\rm P}(\mu)$, then $\mu * \gamma_{0,t}$ satisfies a Poincar\'e inequality with constant $C_{\rm P}(\mu * \gamma_{0,t}) \le C_{\rm P}(\mu) + t$, and the log-Sobolev inequality no longer suffers an explicit exponential dependence on $R^2/t$.
\end{rmk}

\subsection{Extension to sub-Gaussian tails}\label{sec:sg_tails}

Consider the setting in the previous section.
However, we now relax the assumption that $\mu$ has bounded support, and instead assume that $\mu$ has sub-Gaussian tails.
More specifically, assume that there exist constants $\sigma^2, C_{\rm SG}$ such that
\begin{align}\label{eq:sg_tails}
    \iint \exp\bigl(\frac{\norm{x-x'}^2}{\sigma^2}\bigr) \, \D \mu(x) \, \D \mu(x') \le C_{\rm SG}\,.
\end{align}
Since a log-Sobolev inequality implies sub-Gaussian tails~\cite[\S 5.4]{bakrygentilledoux2014}, the existence of such constants $\sigma^2$, $C_{\rm SG}$ are certainly necessary in order for $\mu * \gamma_{0,t}$ to satisfy~\eqref{eq:lsi}. We will show that if $t$ is a large enough multiple of $\sigma^2$, then we indeed obtain a log-Sobolev constant for $\mu * \gamma_{0,t}$, and we will explicitly estimate the constant.

The main point is to estimate, for $X$, $X'$ i.i.d.\ from $\mu$,
\begin{align*}
    \E[{\{1 + \chi^2(\gamma_{X,t} \mmid \gamma_{X',t})\}}^p]
    &= \iint \exp\bigl(\frac{p \, \norm{x-x'}^2}{t}\bigr) \, \D \mu(x) \, \D \mu(x')
    \le C_{\rm SG}\,,
\end{align*}
provided that $t/p \ge \sigma^2$; then, ${K_{p,\; \chi^2}(\PP; \mu)}^{p^*} \le C_{\rm SG}^{p^*/p}$.
We therefore take $p = t/\sigma^2$ and we obtain as an immediate consequence of Theorem~\ref{thm:LS_new} the following result.

\begin{thm}\label{thm:lsi_sg_tails}
    Suppose $\mu$ is a probability measure on $\R^d$ satisfying~\eqref{eq:sg_tails} and that $t > \sigma^2$. Then, $\mu * \gamma_{0,t}$ satisfies both~\eqref{eq:pi} and~\eqref{eq:lsi}, with
    \begin{align*}
        C_{\rm P}(\mu * \gamma_{0,t})
        &\le t\,\bigl\{ \frac{t}{t-\sigma^2} + C_{\rm SG}^{\sigma^2/(t-\sigma^2)}\bigr\}\,,
    \end{align*}
    and
    \begin{align*}
        C_{\rm LS}(\mu * \gamma_{0,t})
        &\le 3t \, \bigl\{ \frac{t}{t-\sigma^2} + C_{\rm SG}^{\sigma^2/(t-\sigma^2)}\bigr\} \, \bigl\{1 + \frac{\sigma^2}{t-\sigma^2} \log C_{\rm SG}\bigr\}\,.
    \end{align*}
\end{thm}

\begin{rmk}
    The first part of Theorem~\ref{thm:lsi_sg_tails} was observed without proof in~\cite{courtade2020poincare}.
\end{rmk}

\begin{rmk}
    The result of Theorem~\ref{thm:lsi_sg_tails} recovers the result of Corollary~\ref{thm:LS_gausssian}, albeit with worse constants. Indeed, if $\mu$ has support contained in the ball $B(0,R)$ and $t > 0$, then we can take $\sigma^2 = t/2$ and
    \begin{align*}
        C_{\rm SG}
        &= \iint \exp\bigl( \frac{2 \, \norm{x-x'}^2}{t}\bigr) \, \D \mu(x) \, \D \mu(x')
        \le \exp \frac{8R^2}{t}\,.
    \end{align*}
    Then, Theorem~\ref{thm:lsi_sg_tails} yields a log-Sobolev inequality for $\mu * \gamma_{0,t}$ with a similar dependence as Corollary~\ref{thm:LS_gausssian}.
\end{rmk}

\begin{rmk}
    The sub-Gaussian tail condition~\eqref{eq:sg_tails} is equivalent to $\mu$ satisfying a $T_1$ transportation-cost inequality~\cite{bolleyvillani2005weightedpinsker}. Hence, our result shows that sufficient Gaussian smoothing upgrades a $T_1$ inequality to a log-Sobolev inequality.
    
    Note that the condition $t > \sigma^2$ is similar to the condition in~\cite[Theorem 1.2]{wang2016functional}.
\end{rmk}

\begin{rmk}
    As in Remark~\ref{rmk:improving_asymp_const}, the Poincar\'e and log-Sobolev constants here can easily be improved when $t\to\infty$ to improve the constant factor in front of $t$ to $1$.
\end{rmk}

\subsection{General diffusions}\label{sect:gen_diff}
We now consider a different extension of the setting in Section~\ref{scn:gaussian_conv}.
Let ${(P^t)}_{t\geq 0}$ be a Markov semigroup on $(\sY,\borelY)$ with invariant measure $\pi$ and infinitesimal generator $\ms L$. Let $\mc A $ be an algebra of bounded measurable functions such that $\mc A$ is dense in $L^2(\sY,\pi)$; $\mc A$ is contained in the domain of $\ms L$; and the \textit{carr\'e du champ} operator $\Gamma:\mc A\times \mc A \to \mc A$ given by
\begin{align*}
    \Gamma (f,g) = \frac{1}{2}\,\bigl(\ms L(fg) - f\ms L g- g \ms L f\bigr)
\end{align*}
is well defined for $f,g\in \mc A$. We assume these objects satisfy the conditions specified in \cite[\S 1.14]{bakrygentilledoux2014} so that results therein are applicable. For $\curv\in\R$ and $t\geq 0$, we set 
\begin{equation*}
    C_{\mathrm{loc}}(\curv,t) :=
    \begin{cases}
    (1-e^{-2\curv t})/\curv,\quad & \curv \neq 0\\
    2t,\quad & \curv = 0.
    \end{cases}
\end{equation*}
We recall the following result (\cite[Theorem 5.5.2]{bakrygentilledoux2014}).
\begin{lem}\label{lemma:loc}
For every $\curv\in\R$, the following statements are equivalent.
\begin{enumerate}
    \item The curvature-dimension condition $\mathrm{CD}(\curv,\infty)$ holds.
    \item For all $x\in E$ and $t\geq 0$,
    \begin{align*}
        C_{\on{LS}}(P^t_x)\leq C_{\on{loc}}(\curv,t)\,.
    \end{align*}
\end{enumerate}
\end{lem}

The following result is then a special case of Theorem~\ref{thm:LS_new}.
\begin{cor}\label{thm:LS_diffusion}
Suppose that the curvature-dimension condition $\on{CD}(\curv,\infty)$ holds for some $\curv\in\R$. Let $\mu$ be a probability measure on $(\sY,\borelY)$. Then, for every $t\geq 0$,
\begin{equation*}
    C_{\on{LS}}(\mu P^t) \leq 6C_{\on{loc}}(\curv,t) \,K_{\infty,\;\chi^2}(P^t; \mu)\, \{1 + \log K_{\infty,\;\chi^2}(P^t; \mu)\}\,.
\end{equation*}

\end{cor}


\begin{rmk}
    If $\sY$ is a complete connected Riemannian manifold and the diffusion has generator $\ms L = \Delta + \langle \nabla V, \nabla\cdot\rangle$ which satisfies the curvature-dimension condition, then under mild conditions the constant $K_{\infty,\;\chi^2}(P^t;\mu)$ is finite for any measure $\mu$ with bounded support, as a consequence of heat kernel estimates in~\cite{gongwang2001heatkernel}.
\end{rmk}

\subsection{Mixtures of two distributions}\label{scn:mixtures_of_two}

In this section, we consider the case when $\sX = \{0,1\}$ is the two-point space. Then, the mixing distribution $\mu$ is a Bernoulli distribution with a mixing weight $p \in [0, 1]$, and the measure $\mu\PP$ is the convex combination
\begin{align}\label{eq:mixture_of_two}
    \mu\PP
    &= (1-p) \PP_0 + p \PP_1\,.
\end{align}

Functional inequalities for such mixtures were studied in~\cite{chafaimalrieu2010mixtures, schlichting2019mixtures}. One of the interesting findings of these papers is that as the mixing weight $p$ tends to $\{0,1\}$, the Poincar\'e constant can remain bounded whereas the log-Sobolev constant diverges logarithmically. Specifically,~\cite{schlichting2019mixtures} shows that if $\PP_0$ and $\PP_1$ satisfy~\eqref{eq:lsi}, $p \in (0,1)$, and either $\chi^2(\PP_0 \mmid \PP_1)$ or $\chi^2(\PP_1 \mmid \PP_0)$ is finite, then $\mu\PP$ satisfies~\eqref{eq:lsi}. Note that this last assumption is weaker than ours, which requires both $\chi^2(\PP_0 \mmid \PP_1)$ and $\chi^2(\PP_1 \mmid \PP_0)$ to be finite.
However, even under our stronger assumption, the bound of~\cite{schlichting2019mixtures} on the log-Sobolev constant diverges in general as $p\to \{0,1\}$.

We now present our results for this setting for comparison.

\begin{cor}
    For all $p \in [0, 1]$, the mixture~\eqref{eq:mixture_of_two} satisfies~\eqref{eq:lsi} with
    \begin{align*}
        C_{\rm LS}(\mu\PP)
        &\le 6 \max\{C_{\rm LS}(\PP_0), C_{\rm LS}(\PP_1)\} \, K_{\chi^2} \, \{1+\log(1+K_{\chi^2})\}\,,
    \end{align*}
    where $K_{\chi^2} := \max\{\chi^2(\PP_0 \mmid \PP_1), \chi^2(\PP_1 \mmid \PP_0)\}$.
\end{cor}

In particular, our assumption $K_{\chi^2} < \infty$ guarantees that the mixture satisfies~\eqref{eq:lsi} with a constant independent of $p$, and hence does not exhibit a logarithmic divergence as $p\to \{0,1\}$. We refer to the aforementioned papers for further discussion and examples of mixtures.

\subsection{Analogues on the hypercube}\label{scn:hypercube}

We now present another interesting illustration of our results.
Here, we take $\sX = {\{0,1\}}^n$ to be the Boolean hypercube, and we take $\sY := \eu Y^n$ to be a product space. We also require the $\Gamma$ operator on $\sY$ to be consistent with the product structure; for simplicity of presentation, we omit this discussion and instead think of $\Gamma$ as being either the squared gradient operator $\Gamma(f) = \norm{\nabla f}^2$ on Euclidean space, or the discrete gradient $\Gamma(f) = {(Df)}^2$ as described in Section~\ref{scn:background}.
Let $\pi_0$, $\pi_1$ be two probability measures on $\eu Y$ with
\begin{align*}
    K_{\rm LS}(\pi)
    &:= \max\{C_{\rm LS}(\pi_0), C_{\rm LS}(\pi_1)\} < \infty\,, \\
    K_{\chi^2}(\pi)
    &:= \max\{\chi^2(\pi_0 \mmid \pi_1), \chi^2(\pi_1 \mmid \pi_0)\} < \infty\,.
\end{align*}
Given $x \in {\{0,1\}}^n$, define the measure
\begin{align}\label{eq:hamming_prod_measure}
    \PP_x
    &= \bigotimes_{i=1}^n \pi_{x_i}\,.
\end{align}
From the tensorization of the chi-squared divergence,
\begin{align*}
    \chi^2(P_x \mmid P_{x'})
    &= \prod_{i=1}^n \{1+\chi^2(\pi_{x_i} \mmid \pi_{x_i'})\} - 1
    \le {\{1 + K_{\chi^2}(\pi)\}}^{d(x,x')} - 1\,,
\end{align*}
where $d(\cdot,\cdot)$ denotes the Hamming metric on ${\{0,1\}}^n$. Moreover, each $\PP_x$ satisfies~\eqref{eq:lsi} with a constant at most $K_{\rm LS}(\pi)$, due to the classical tensorization of log-Soboblev inequalities. As a consequence, we deduce the following result from Theorem~\ref{thm:LS_new}.

\begin{cor}
    Suppose $\mu$ is a probability measure on ${\{0,1\}}^n$ which is supported on a set of diameter at most $k$ in the Hamming metric.
    Then, the mixture distribution $\mu\PP := \sum_{x\in {\{0,1\}}^n} \mu(x) \PP_x$, with $\PP_x$ as in~\eqref{eq:hamming_prod_measure}, satisfies~\eqref{eq:lsi} with
    \begin{align*}
        C_{\rm LS}(\mu\PP)
        &\le 6k \, K_{\rm LS}(\pi) \, {\{1 + K_{\chi^2}(\pi)\}}^k \,\bigl\{1 + \log\bigl(1 + K_{\chi^2}(\pi)\bigr)\bigr\}\,.
    \end{align*}
\end{cor}
Importantly, the log-Sobolev inequality is dimension-free in the sense that it depends only on properties of $\pi_0$ and $\pi_1$ as well as the diameter $k$ of the support of $\mu$. An example of such a measure $\mu$ is any measure which is supported on $k/2$-sparse strings.

We now specialize this result to obtain an analogue of the result for Gaussian convolutions in Section~\ref{scn:gaussian_conv} to the setting of the Boolean hypercube. Let $0 < p < 1/2$, and we take $\pi_0$ and $\pi_1$ to be the Bernoulli distributions with parameters $p$ and $1-p$ respectively. Also, we take the $\Gamma$ operator to be the square of discrete gradient. The optimal log-Sobolev inequality for these distributions is given in~\cite[Problem 8.3]{vanhandel2018probability}, and a quick computation yields
\begin{align*}
    K_{\rm LS}(\pi)
    &= \frac{p \, (1-p)}{2 \, (1-2p)} \log \frac{1-p}{p}\,, \qquad K_{\chi^2}(\pi) = \frac{{(1-p)}^2}{p} + \frac{p^2}{1-p}-1\,.
\end{align*}
Note that the mixture $\mu\PP$ can be interpreted as the result of evolving the initial measure $\mu$ for a short time under the natural semigroup on the hypercube. We obtain the following result.

\begin{cor}
    Suppose $\mu$ is a probability measure on ${\{0,1\}}^n$ which is supported on a set of diameter at most $k$ in the Hamming metric.
    Then, the mixture distribution $\mu\PP$ with $0 < p < 1/2$ satisfies~\eqref{eq:lsi} with
    \begin{align*}
        C_{\rm LS}(\mu\PP)
        &\le \frac{6k}{p^{k-1} \, (1-2p)} \log^2 \frac{1}{p}\,.
    \end{align*}
\end{cor}

\section*{Acknowledgments}

Sinho Chewi was supported by the Department of Defense (DoD) through the National Defense Science \& Engineering Graduate Fellowship (NDSEG) Program. Jonathan Niles-Weed was supported in part by National Science Foundation (NSF) grant DMS-2015291.

\appendix

\section{Tightening of LSI}\label{scn:tightening_lsi}

The following proposition is a standard result, see~\cite[Proposition 5.1.3]{bakrygentilledoux2014}. It is straightforward to see that bilinearity of $\Gamma$ and our assumption~\eqref{eq:Gamma_prop} are sufficient for the proof to go through.


\begin{prop}\label{prop:tighten}
\leavevmode
\begin{enumerate}
    \item \label{item:prop_1} If $\rho$ satisfies~\eqref{eq:lsi}, then $\rho$ satisfies~\eqref{eq:pi} with $C_{\on{P}}(\rho)\leq C_{\on{LS}}(\rho)$.  
    \item \label{item:prop_2} If $\rho$ satisfies the following defective LSI
    \begin{align*}
        \ent_\rho(f^2)\leq 2C\E_\rho\Gamma(f)+D\E_\rho(f^2)\,\qquad\forall f\in\mc A\,,
    \end{align*}
    together with~\eqref{eq:pi}, then $\rho$ satisfies~\eqref{eq:lsi} with
    \begin{align*}
        C_{\rm LS}(\rho)
        &\le C+C_{\rm P}(\rho) \,\bigl(\frac{D}{2}+1\bigr)\,.
    \end{align*}
\end{enumerate}
\end{prop}

\RaggedRight{}
\printbibliography{}

\end{document}